\documentclass[11pt]{article}

\usepackage[T1]{fontenc}
\usepackage{lmodern}
\usepackage{microtype}
\usepackage{amsmath,amssymb,amsthm,mathtools}
\usepackage{enumitem}
\usepackage{needspace}
\usepackage{aliascnt}
\usepackage{geometry}
\usepackage[hidelinks]{hyperref}
\hypersetup{
  pdftitle={Bistationary Traces, Wide Levels, and Branch-Cover Rigidity for an Unrestricted Typed Variant of the Hayut--Magidor Forcing},
  pdfauthor={Xing-Yu Hu},
  pdfsubject={Set theory and forcing},
  pdfkeywords={bistationary traces, ladder-coordinate sets, two-cardinal trees, forcing, strategic closure, branch-covering numbers, inaccessible cardinals}
}
\usepackage[nameinlink,capitalize,noabbrev]{cleveref}
\newcommand{\PP}{\mathcal P}
\newcommand{\cf}{\operatorname{cf}}
\newcommand{\forces}{\Vdash}
\newcommand{\Br}{\operatorname{Br}}
\newcommand{\bcov}{\operatorname{bcov}}
\newcommand{\Tr}{\operatorname{Tr}}
\newcommand{\domn}{\operatorname{dom}}
\newcommand{\rng}{\operatorname{range}}
\newcommand{\Sstar}{\mathbb S^{\ast}}

\newtheorem{theorem}{Theorem}[section]
\newaliascnt{lemma}{theorem}
\newtheorem{lemma}[lemma]{Lemma}
\aliascntresetthe{lemma}
\newaliascnt{proposition}{theorem}
\newtheorem{proposition}[proposition]{Proposition}
\aliascntresetthe{proposition}
\newaliascnt{corollary}{theorem}
\newtheorem{corollary}[corollary]{Corollary}
\aliascntresetthe{corollary}
\theoremstyle{definition}
\newaliascnt{definition}{theorem}
\newtheorem{definition}[definition]{Definition}
\aliascntresetthe{definition}

\crefname{theorem}{Theorem}{Theorems}
\crefname{lemma}{Lemma}{Lemmas}
\crefname{proposition}{Proposition}{Propositions}
\crefname{corollary}{Corollary}{Corollaries}
\crefname{definition}{Definition}{Definitions}

\title{Bistationary Traces, Wide Levels, and Branch-Cover Rigidity\\
       for an Unrestricted Typed Variant of the Hayut--Magidor Forcing}
\author{Xing-Yu Hu\thanks{Corresponding author.}\\
\small School of Mathematics and Statistics, Hanjiang Normal University}
\date{}

\begin{document}
\maketitle

\begin{abstract}
For every uncountable regular cardinal $\alpha$, $\Sstar(\alpha)$ is an
explicitly typed four-coordinate forcing motivated by the ladder-system
construction of Hayut and Magidor.  The forcing is $\sigma$-closed and, after
adjoining a formal maximum, $\alpha$-strategically closed.  For regular $\alpha\geq\omega_2$, every nonempty family of fewer than $\alpha$
designated generic branches has a stationary and costationary common
trace in $\PP_{\omega_1}\alpha$ with respect to the generic
ladder-coordinate set $L_\alpha$.  For these cardinals, the minimum number of
cofinal branches of the generic tree needed to generate $L_\alpha$
is exactly $\alpha$.
The lower bound follows from a uniform diagonal argument, whereas
the upper bound uses one designated branch for each ladder coordinate with a
nonempty value.
For the Kurepa-bounded companion forcing at the same cardinals, the bistationary-trace
conclusion holds, and the corresponding branch-covering number is also
exactly $\alpha$.
In the unrestricted forcing at such an $\alpha$, for every infinite cardinal
$\mu<\alpha$ in the ground model, some level of the generic tree contains a copy
of $({}^\mu2)^V$.  The endpoint-corrected restriction family
indexed by $\PP_{\omega_2}\alpha$ has a level of size at least $\omega_2$
and hence is not a $\PP_{\omega_2}\alpha$-tree, whereas the scaled restriction
system indexed by $\PP_\alpha\alpha$ has all levels of size less than $\alpha$
exactly when $\alpha$ is strongly inaccessible in the ground model.  When
these equivalent conditions hold, the branch-covering number of $L_\alpha$
relative to the scaled system equals $\alpha$.  For regular $\alpha\geq\omega_2$,
the empty-value requirement at ladder coordinates of cofinality at most
$\omega$ also prevents the set of domains of $L_\alpha$ from containing a club
in $\PP_\alpha\alpha$.  As in Hayut and Magidor's tree clause, the unrestricted
forcing places no explicit size bound on the levels of a condition's tree.
No forcing equivalence is asserted.
\end{abstract}



\section{Introduction}

Simultaneous stationary traces and branch-cover obstructions interact with
level-size constraints in a typed forcing motivated by ladder systems on
two-cardinal trees.  For a ladder system on a two-cardinal tree, a
cofinal branch may meet the ladder cofinally, stationarily, or on a club.  Hayut and Magidor proved,
relative to a supercompact cardinal, a separation between $\omega_1$-cofinal and
$\omega_1$-club catching at $\omega_2$
\cite[Theorem~5.2, p.~1121]{HayutMagidor2022}, and asked whether cofinal
catching can be separated from stationary catching
\cite[Question~6.1, p.~1128]{HayutMagidor2022}.  Their four-coordinate
construction combines a binary ordinal tree, coherent designated branches, a
generic ladder, and an auxiliary coordinate
\cite[Definition~5.5(1)--(4) and Notation~5.6, p.~1123]{HayutMagidor2022}.

The notion of a $\PP_\kappa\lambda$-tree goes back to Jech's
strong-compactness characterization \cite[Section~2]{Jech1973}, while Magidor
characterized supercompactness by the corresponding ineffability principles on
$\PP_\kappa\lambda$ for every $\lambda\geq\kappa$
\cite[p.~282]{Magidor1974}.  Fontanella obtained relative-consistency results for the super tree
property at small successor cardinals \cite{Fontanella2012,Fontanella2013}
and for the strong tree property at successors of singular cardinals
\cite{Fontanella2014}.  Wei{\ss} introduced combinatorial principles characterizing strong
compactness and supercompactness at inaccessible cardinals, with formulations
that also apply at successor cardinals \cite{Weiss2012}.  Viale and
Wei{\ss} proved that the Proper Forcing Axiom implies the corresponding
principles at $\omega_2$ \cite{VialeWeiss2011}.  The interaction of these generalized tree principles with the singular
cardinal hypothesis was analyzed in \cite{HachtmanSinapova2019}, while their consistency
at successors of singular cardinals was studied further in
\cite{HachtmanSinapova2020,Adkisson2024}.  Lambie-Hanson and Stejskalov\'a connected generalized tree properties
with guessing models and weak Kurepa trees
\cite{LambieHansonStejskalova2024}, and later studied cardinal-arithmetic
consequences of guessing-model principles
\cite{LambieHansonStejskalova2026}.  Two-cardinal
Kurepa families give a related bound on the number of restrictions to small sets.  For every uncountable strong
limit $\kappa$ and $\lambda\geq\kappa$, there is a family of
$2^\lambda$ subsets of $\lambda$ with fewer than $\kappa$ distinct
restrictions to each $x\in\PP_\kappa\lambda$
\cite[Definition~2.1(i)--(ii) and Proposition~2.2(iii), p.~4]{Wu2025}.

Two features of the published presentation require explicit choices in a typed
formulation.  First, the displayed type of the auxiliary coordinate and the
types used in the closure and no-club-catching arguments do not coincide
\cite[Definition~5.5(4), p.~1123, and Claims~5.7--5.8, p.~1124]{HayutMagidor2022}.  Second, the displayed tree clause states
no explicit level-size bound on a condition tree, although the generic object is
described as an $\alpha$-tree and Claim~5.7 states a $2^{<\alpha}$ bound for the
size of the forcing
\cite[Definition~5.5(1), Notation~5.6, and Claim~5.7, p.~1123]{HayutMagidor2022}.  Accordingly, $\Sstar(\alpha)$ is an explicitly typed variant with no
explicit level-size bound and with its normality convention fixed below.
Its auxiliary coordinate is a partial function on
countable sets.  A non-dummy value at $z$ is a node of the restriction level
indexed by $z$, and defined ladder coordinates of cofinality at most $\omega$ have
the empty value.  No forcing equivalence with the printed presentation is
claimed.  The precise comparison is given in Subsection~\ref{sec:relation}.
Unless explicitly stated otherwise, cardinal and regularity assumptions on
$\alpha$ refer to the ground model.

The calculations below distinguish the effects of unrestricted tree width
from those that persist under a Kurepa level bound.  For regular $\alpha\geq\omega_2$,
the intersection of every nonempty family of fewer than $\alpha$ designated
traces is also bistationary in the Kurepa-bounded companion extension.
In their respective generic extensions, both the unrestricted
and Kurepa-bounded versions have branch-covering number exactly $\alpha$.
The absence of a level-size bound permits the wide-level construction
used to determine the strong-limit boundary for the scaled restriction system.

Let $T_\alpha$, $L_\alpha$, and $B_\alpha(u)$ denote the generic tree,
ladder-coordinate set, and designated branches added by $\Sstar(\alpha)$.
For $b\in{}^\alpha2$, put
$\Tr_{\omega_1}(b,L_\alpha)=\{x\in\PP_{\omega_1}\alpha:b\restriction x\in L_\alpha\}$,
and say that $\mathcal B\subseteq{}^\alpha2$ generates $L_\alpha$ when each
$y\in L_\alpha$ is extended by some $b\in\mathcal B$.
\Needspace{10\baselineskip}
\begin{theorem}\label{thm:main}
Let $\alpha\geq\omega_2$ be regular and let
$G\subseteq\Sstar(\alpha)$ be generic.  Then $\alpha$ remains a regular
cardinal in $V[G]$, and the following assertions hold.
\begin{enumerate}[label=\textup{(\roman*)},leftmargin=2.2em]
  \item for every nonempty $U\subseteq T_\alpha$ with $|U|<\alpha$, the set
  \[
     \bigcap_{u\in U}
     \Tr_{\omega_1}(B_\alpha(u),L_\alpha)
  \]
  is stationary and costationary in $\PP_{\omega_1}\alpha$.
  \item the least cardinality of a family of cofinal branches through
  $T_\alpha$ generating $L_\alpha$ is exactly $\alpha$.
  \item for every infinite cardinal $\mu<\alpha$ in $V$, there are a limit
  ordinal $\delta<\alpha$ and an injection
  \[
     ({}^\mu2)^V\longrightarrow (T_\alpha)_\delta.
  \]
  \item the endpoint-corrected Hayut--Magidor-style restriction family
  \[
     \mathcal T^{\omega_2}_{\alpha,x}
       =\{r\restriction x:r\in(T_\alpha)_{\sup(x)+1}\}
       \qquad(x\in\PP_{\omega_2}\alpha)
  \]
  has a level of size at least $\omega_2$ and hence is not a
  $\PP_{\omega_2}\alpha$-tree.
\end{enumerate}
For $x\in\PP_\alpha\alpha$, put
$\mathcal T_{\alpha,x}=\{r\restriction x:r\in(T_\alpha)_{\sup(x)+1}\}$,
and write $\mathcal T_\alpha$ for this scaled restriction system.  Its levels
have size less than $\alpha$ if and only if $\alpha$ is strongly inaccessible
in $V$.  Under these equivalent conditions, $\mathcal T_\alpha$ is a
$\PP_\alpha\alpha$-tree, and its branch-covering number for $L_\alpha$
equals $\alpha$.  In particular, if $\alpha$ is a successor cardinal in $V$, then
the scaled restriction system is not a $\PP_\alpha\alpha$-tree.
\end{theorem}

By \cref{prop:level-domain-failure}, the empty-value convention
prevents $L_\alpha$ from satisfying the club-many-level requirement
of a two-cardinal ladder system, independently of unrestricted tree width.
Simultaneous stationarity uses a dummy value to reserve a countable domain
forced into a named club.  Anchored insertion gives a dense open requirement
for each designated seed below the condition reserving that dummy.
By $<\alpha$-distributivity, fewer than $\alpha$ such requirements can be
met simultaneously.
Costationarity comes from the complementary non-dummy value and
\cref{prop:no-club}.  Thus every nonempty intersection of fewer than
$\alpha$ designated traces is bistationary.

The exact branch-covering number follows from a uniform diagonal against
fewer than $\alpha$ cofinal branches and an upper bound supplied by the
ladder coordinates themselves.  The common-trace conclusion and the exact covering number
survive the Kurepa-style level bound of
\cref{prop:kurepa-bounded-persistence}.  The unrestricted tree-extension
construction gives levels containing copies of $({}^\mu2)^V$ for every infinite cardinal
$\mu<\alpha$ in $V$.  Consequently, the restriction family indexed by
$\PP_{\omega_2}\alpha$ has a level of size at least $\omega_2$ for regular
$\alpha\geq\omega_2$.  For the scaled system indexed by
$\PP_\alpha\alpha$, all restriction levels are small precisely when $\alpha$
is strongly inaccessible in $V$.

The tree and designated-branch coordinates, together with the limit-stage
closure argument, are adapted from the Hayut--Magidor construction
\cite[Definition~5.5(1)--(3) and Claim~5.7, pp.~1123--1124]{HayutMagidor2022}.
Their closure proof uses dummy auxiliary values, while their no-club proof
uses non-dummy values to exclude a decided branch restriction
\cite[Claims~5.7--5.8, p.~1124]{HayutMagidor2022}.  In Claim~5.7, Even
is passive at successor stages \cite[p.~1123]{HayutMagidor2022}.  In our
strategy, Even instead reserves a fresh dummy supremum at each positive
successor stage assigned to her.  For the typed forcing, the simultaneous
stationary-trace argument combines a reserved dummy domain, anchored
insertion, and $<\alpha$-distributivity.  For costationarity, we adapt the no-club
obstruction.  The coordinatewise upper bound for the exact branch-covering
number uses the stipulated empty values at low cofinalities.  The companion
uses the Kurepa level bound from
\cite[Example~6.1, p.~794]{Cummings2010}.

The conclusions for $\Sstar(\alpha)$ concern its direct generic extension.
They do not settle the question posed by Hayut and Magidor
\cite[Question~6.1, p.~1128]{HayutMagidor2022}.  Immediately before that
question, Hayut and Magidor note that their witnessing branch already meets
the ladder system on a stationary rather than merely unbounded set
\cite[p.~1128]{HayutMagidor2022}.  Here every nonempty intersection
of fewer than $\alpha$ designated traces is bistationary, and the generic
ladder-coordinate set has branch-covering number exactly $\alpha$.

The hypotheses on $\alpha$ have distinct roles.  Ground-model regularity
keeps all strategic limit constructions below $\alpha$, and
$\alpha$-strategic closure preserves $\alpha$ as a regular cardinal.  The
additional bound $\alpha\geq\omega_2$ ensures that a ladder coordinate of
cofinality $\omega_1$ can be chosen above any prescribed condition height.
The unrestricted tree clause makes the Hayut--Magidor-style restriction
family too wide and gives the exact strong-limit boundary for the scaled
$\PP_\alpha\alpha$ system.

\section{The Typed Forcing and Branch Covers}

\subsection{Relation to the Hayut--Magidor construction}\label{sec:relation}

The published presentation motivates the typing choice for the $f$-coordinate.
Definition~5.5(4) requires pairs $(z,r)$ with
$r\in t_{\sup z}\cup\{-1\}$
\cite[Definition~5.5(4), p.~1123]{HayutMagidor2022}.  Such a
non-dummy value has ordinal domain.  In contrast, the proof of Claim~5.8 uses the assignment
$f^q(M\cap\delta)=x$ with $\domn(x)=M\cap\delta$, while the proof of
Claim~5.7 writes $f(y)$ for a ladder restriction $y$
\cite[p.~1124]{HayutMagidor2022}.

A non-dummy value at a countable set $z$ is the restriction to $z$ of an
ordinal-tree node.  This makes both $\rng(f)\cap\ell(\xi)$ and
$f(M\cap\delta)=x$ well typed.  Accordingly, the expression $f(y)$ in that proof is interpreted as
$f(\domn(y))$.  At defined ladder coordinates of cofinality at most $\omega$, this
variant requires the empty value.  Definition~5.5(2) prescribes the displayed ladder form only at
uncountable cofinality
\cite[Definition~5.5(2), p.~1123]{HayutMagidor2022}, while the proof of
Claim~5.7 uses the empty value at limit stages of cofinality $\omega$ in its closure construction
\cite[p.~1124]{HayutMagidor2022}.

For its $f$-coordinate, the published extension relation requires end
extension alone \cite[Definition~5.5, p.~1123]{HayutMagidor2022}.
The typed variant additionally requires every new $f$-domain
to have supremum above the previous tree height.  This convention is used in the support and disjointness arguments at
limit stages.

The displayed tree clause in
\cite[Definition~5.5(1), p.~1123]{HayutMagidor2022} states no explicit
level-size bound on a condition tree.  At the same time, the surrounding
discussion describes the generic object as an $\alpha$-tree in the
$\PP_{\omega_2}\alpha$ sense fixed immediately before that definition
\cite[p.~1122]{HayutMagidor2022}, and
\cite[Claim~5.7, p.~1123]{HayutMagidor2022} states that the forcing has
size $2^{<\alpha}$.  If $\alpha$ is strongly inaccessible, every condition level
nevertheless has size less than $\alpha$, since
$|(t_p)_\xi|\leq2^{|\xi|}<\alpha$.  At successor cardinals, the displayed
condition-tree clause alone gives no analogous smallness bound.  The present
forcing likewise imposes no explicit level-size bound, but uses the normality
convention stated in Subsection~\ref{sec:preliminaries}.  The unconditional
wide-level and exact strong-limit conclusions in
\cref{lem:wide-level,prop:two-cardinal-boundary} depend on the absence of
this bound.  These comparisons do not establish forcing equivalence
with the printed presentation.

Cummings's Kurepa-tree forcing at an inaccessible cardinal imposes
the level bound
$|T_\xi|\leq|\xi|+\omega$ \cite[Example~6.1, p.~794]{Cummings2010}.
\Cref{prop:kurepa-bounded-persistence} shows that the bistationary-trace
conclusion and the exact branch-covering number persist for the companion
forcing obtained by restricting the typed conditions to this bound.
Thus the trace and cover results are not consequences
of unrestricted width.

Hayut and Magidor describe a restriction system for ordinals
$\alpha\leq\omega_2$ whose levels are indexed by $\sup(x)$
\cite[p.~1122]{HayutMagidor2022}.
For regular $\alpha\geq\omega_2$, the endpoint-corrected analogue is
$\beta_x=\sup(x)+1$.  The successor is necessary when $x$ has a maximum.

The $b$-coordinate retains the designated-branch construction of Hayut and Magidor
\cite[Definition~5.5(3) and Notation~5.6, p.~1123]{HayutMagidor2022}.
At limits of cofinality $\omega$, the strategy reserves a dummy $f$-domain
at the new supremum.  At a limit of uncountable cofinality, the new ladder
club is chosen so that the supremum of each restriction domain is a previously
reserved dummy supremum.  If a ladder restriction equaled a non-dummy
$f$-value, equality of domains and uniqueness of domain suprema would force
that $f$-value to be the corresponding dummy, a contradiction.
The auxiliary coordinate plays complementary roles in
\cref{thm:stationary-traces,prop:no-club}.  A non-dummy value excludes a
decided restriction, while a dummy value leaves room for a later insertion.

Assuming that $\kappa$ is $\lambda$-$\Pi^1_1$-subcompact,
Lemma~5.11 establishes the cofinal ladder-catching principle in the
$\mathbb{M}'(\kappa)*\dot{\mathbb{S}}(\kappa)$-generic extension.
Given a $\PP_{\omega_2}\lambda$-tree and a ladder system in that extension,
the proof first constructs a branch in a further generic extension
using a lifted elementary embedding
\cite[Lemma~5.11, pp.~1125--1126]{HayutMagidor2022}.  Claim~5.12 assumes that in the intermediate
extension $V[G]$ no cofinal branch through the given tree meets its ladder
system $\omega_1$-cofinally.  It then shows that the particular subsequent
forcing considered there adds no such branch
\cite[Claim~5.12, p.~1126]{HayutMagidor2022}.  The final paragraph of the proof of Lemma~5.11
\cite[p.~1127]{HayutMagidor2022} shows that the lifted-model trace contains a club.
This does not determine the traces of the designated family considered here.

\subsection{Definition, closure, and branch covers}\label{sec:preliminaries}

For a cardinal $\rho$ and a set $X$, let
\[
  \PP_\rho X=\{x\subseteq X:|x|<\rho\}.
\]
A subset of $\PP_\rho X$ is \emph{club} if it is unbounded under inclusion
and closed under increasing unions of length less than $\cf(\rho)$.  This is
the convention stated immediately before Definition~4.1 of
\cite[p.~1119]{HayutMagidor2022}.  Throughout this paper, the increasing
sequences in this closure clause are understood to have positive length.  A set
is \emph{stationary} if it meets
every club.  For background on generalized stationary sets, see
\cite[pp.~93--128]{JechStationary2010}.

For cardinals $\kappa\leq\lambda$, a $\PP_\kappa\lambda$-tree is a
system
\[
   \mathcal T=\langle\mathcal T_x:x\in\PP_\kappa\lambda\rangle
\]
such that $\varnothing\neq\mathcal T_x\subseteq{}^x2$, restrictions from
larger levels belong to smaller levels, and
$|\mathcal T_x|<\kappa$ for every $x$.  A branch is a function
$b:\lambda\to2$ satisfying $b\restriction x\in\mathcal T_x$ for every
$x\in\PP_\kappa\lambda$.  This is Jech's notion
\cite[Section~2]{Jech1973}, in the formulation used in
\cite[Definition~4.1, p.~1119]{HayutMagidor2022}.

A \emph{ladder system} on $\mathcal T$ is a set
$L\subseteq\bigcup_{x\in\PP_\kappa\lambda}\mathcal T_x$ such that
$L\cap\mathcal T_x\neq\varnothing$ for club many $x$, and whenever
$\eta\in L\cap\mathcal T_x$ and
$\cf(|x\cap\kappa|)>\omega$, there is a club
$E_\eta\subseteq\PP_{|x\cap\kappa|}x$ satisfying
$\eta\restriction y\in L$ for every $y\in E_\eta$
\cite[Definition~4.2, p.~1119]{HayutMagidor2022}.

Let $\rho$ be an infinite cardinal and let $\lambda$ be an ordinal.  For a
function $b\in{}^\lambda2$ and a set $L$ of partial functions on $\lambda$,
put
\[
   \Tr_\rho(b,L)=
   \{x\in\PP_\rho\lambda:b\restriction x\in L\}.
\]
The function $b$ meets $L$ $\rho$-cofinally if this trace is cofinal under
inclusion, and meets $L$ on a $\rho$-club if the trace contains a club.  Hayut
and Magidor formulate the corresponding catching notions for a regular cardinal
$\rho\leq\kappa$ immediately after Definition~4.2
\cite[p.~1119]{HayutMagidor2022}.  In that scope, the cofinal notions agree when
$\rho=\kappa$.  For $\rho<\kappa$, the definition used here is stronger because the source
definition does not require its witness to lie in $\PP_\rho\lambda$.  The final paragraph of the proof of Lemma~5.11
\cite[p.~1127]{HayutMagidor2022} nevertheless proves the relevant
$\PP_{\omega_1}\lambda$-trace unbounded, so
that application satisfies the small-domain form used here.  A trace is
\emph{bistationary} if it and its complement are stationary.  In
particular, a bistationary trace is stationary but contains no club.

For a set $S$ of ordinals, ``nowhere stationary'' means that
$S\cap\beta$ is nonstationary in $\beta$ for every $\beta$ of
uncountable cofinality.  This meaning is used for the support of the
$f$-coordinate below.

Throughout, $-1$ denotes a formal symbol that is not a binary function.
Binary sequences are ordered by extension.  A tree
$t\subseteq{}^{\leq\gamma}2$ is \emph{normal} if it is nonempty,
downward closed, and every node has an extension on every higher level below
$\gamma+1$.  No requirement of two immediate successors at each
nonterminal node is imposed.

For such a tree $t$ and $z\in\PP_{\omega_1}\gamma$, let
\[
   \widehat z=\sup\{\xi+1:\xi\in z\}\leq\gamma,
   \qquad
   t[z]=\{r\restriction z:r\in t_{\widehat z}\}.
\]
Put $\beta_z=\sup(z)+1$.  If $\beta_z\leq\gamma$, then normality and
downward closure give
\[
   t[z]=\{r\restriction z:r\in t_{\beta_z}\}.
\]
The only remaining case is that $z$ has no maximum and is cofinal in $\gamma$.
Then $\widehat z=\gamma$ and $\beta_z=\gamma+1$, and every normal end
extension $t'$ of $t$ whose top height is at least $\beta_z$ satisfies
\[
   t[z]=\{r\restriction z:r\in t'_{\beta_z}\}.
\]
Thus $t[z]$ agrees with the corresponding restriction level of the generic tree,
while $\widehat z$ keeps the definition internal to the present condition.

\Needspace{5\baselineskip}
\begin{definition}\label{def:typed-forcing}
Let $\alpha$ be an uncountable regular cardinal.  A condition in $\Sstar(\alpha)$ is a tuple
$p=\langle t_p,\ell_p,b_p,f_p\rangle$ with tree height
$\gamma_p+1<\alpha$ satisfying the following clauses.
\begin{enumerate}[label=\textup{(\arabic*)},leftmargin=2.2em]
  \item $t_p\subseteq{}^{\leq\gamma_p}2$ is a normal binary tree.
  \item $\ell_p$ is a function whose domain is closed in $\gamma_p+1$.
  For every $\xi\in\domn(\ell_p)$, the following requirements hold.
  If $\cf(\xi)\leq\omega$, then $\ell_p(\xi)=\varnothing$.
  If $\cf(\xi)>\omega$, then there are a node
  $x\in(t_p)_\xi$ and a club
  $E_x\subseteq\PP_{\omega_1}\xi$ such that
  \[
       \ell_p(\xi)=\{x\}\cup\{x\restriction z:z\in E_x\}.
  \]
  \item $b_p:t_p\to(t_p)_{\gamma_p}$ and $v\subseteq b_p(v)$ for every
  $v\in t_p$.
  \item $f_p$ is a partial function with
  $\domn(f_p)\subseteq\PP_{\omega_1}\gamma_p$.  For
  $z\in\domn(f_p)$,
  \[
       f_p(z)\in t_p[z]\cup\{-1\}.
  \]
  Distinct members of $\domn(f_p)$ have distinct suprema, and
  \[
       \{\sup z:z\in\domn(f_p),\ f_p(z)\ne-1\}
  \]
  is nowhere stationary.
  \item For every $\xi\in\domn(\ell_p)$,
  \[
       \rng(f_p)\cap\ell_p(\xi)=\varnothing.
  \]
\end{enumerate}
For conditions $p$ and $q$, write $q\leq p$ when the following clauses hold.
\begin{enumerate}[label=\textup{(\alph*)},leftmargin=2.2em]
  \item $t_q\cap{}^{\leq\gamma_p}2=t_p$.
  \item $\ell_q$ extends $\ell_p$ and
  $\domn(\ell_q)\cap(\gamma_p+1)=\domn(\ell_p)$.
  \item $b_p(v)\subseteq b_q(v)$ for every $v\in t_p$.
  \item $f_q$ extends $f_p$, and every
  $z\in\domn(f_q)\setminus\domn(f_p)$ satisfies
  $\sup z>\gamma_p$.
\end{enumerate}
\end{definition}

The displayed forms of the $t$- and $b$-coordinates and the clause for
$\ell$ at coordinates of uncountable cofinality are modeled on
\cite[Definition~5.5(1)--(3), p.~1123]{HayutMagidor2022}.  The normality
convention used here is the one explicitly stated above.  Clause~(2)
requires the empty value at defined ladder coordinates of cofinality at most $\omega$.  For cofinality
$\omega$, this is the value used in the countable-limit construction in the proof of
Claim~5.7, whose final calculation evaluates $f$ at the domain of the ladder
restriction \cite[p.~1124]{HayutMagidor2022}.  Clause~(4) gives a
type-correct version of the assignment in the proof of Claim~5.8
\cite[p.~1124]{HayutMagidor2022}.  In particular, every
non-dummy value $f_p(z)$ has domain exactly $z$.  Consequently, if
$z\subseteq\domn(x)$ and $f_p(w)=x\restriction z$, then $w=z$.

\begin{theorem}\label{thm:closure}
The forcing $\Sstar(\alpha)$ from \cref{def:typed-forcing} is $\sigma$-closed.  After adjoining a formal
maximum condition, it is $\alpha$-strategically closed.  Consequently,
$\Sstar(\alpha)$ is $<\alpha$-distributive and preserves $\alpha$ as a
regular cardinal.
\end{theorem}

\begin{proof}
For $\sigma$-closure, let
$\langle p_n:n<\omega\rangle$ be descending.  If the tree heights are
eventually constant, then the sequence itself is eventually constant by the
end-extension clauses.  Otherwise, pass to a cofinal subsequence with strictly
increasing heights, and let $\delta$ be their supremum.  Regularity of
$\alpha$ gives $\delta<\alpha$, and $\cf(\delta)=\omega$.  Take the union of
the old trees and add the coherent $b$-limits as a new top level.  For an old
node $v$, let its new $b$-value be the corresponding coherent limit, and let
each new top node be its own $b$-value.

End extension makes the unions of the old $\ell$- and $f$-coordinates well
defined.  The union of the old $\ell$-domains is closed below $\delta$.  If
$\beta<\delta$ is a limit point of this union, choose $n$ whose tree height
exceeds $\beta$.  The old domain at stage $n$ is already closed and contains
every point of the eventual union below $\beta$.  Extend the $\ell$-coordinate
by setting $\ell(\delta)=\varnothing$, whether or not $\delta$ is a limit point
of the old $\ell$-domains.  Choose a countable set $d\subseteq\delta$ cofinal
in $\delta$ and set $f(d)=-1$.  If a prescribed
$a\in\PP_{\omega_1}\delta$ is cofinal in $\delta$, one may take $d=a$.
The new $f$-domain has supremum $\delta$, strictly larger than
the supremum of every old $f$-domain.  Thus the new supremum is fresh and
clause~(d) of \cref{def:typed-forcing} holds.  Let $S$ be the union
of the nontrivial supports.  For every $\beta<\delta$, the set $S\cap\beta$
eventually agrees with the corresponding initial segment of one old support.
There is no requirement at $\delta$, whose cofinality is $\omega$, and above
$\delta$ the support is bounded.  Hence the resulting condition is a common
lower bound.

The strategic-closure argument begins with the limit-stage scheme in the proof of
Claim~5.7 \cite[p.~1124]{HayutMagidor2022}.  There,
Even is passive at successor stages \cite[p.~1123]{HayutMagidor2022}.  In the typed strategy,
Even reserves a fresh dummy supremum at each positive successor stage
assigned to her.  These suprema are used to verify disjointness at limits of
uncountable cofinality.  Adjoin a
formal maximum $\mathbf 1$ to $\Sstar(\alpha)$.  Since the original forcing is
dense in the enlarged poset, the generic extensions are unchanged.  Let
$G_\alpha(P)$ denote the game from
\cite[Definition~5.14, p.~793]{Cummings2010}.  Odd plays at odd stages,
Even plays at even stages (including all limit stages), and Even plays
$\mathbf 1$ at stage~$0$.  If Odd repeats $\mathbf 1$ at stage~$1$, Even
first plays any genuine condition of top height $c$ with $\cf(c)=\omega$,
$\ell(c)=\varnothing$, and $f(d_c)=-1$ for some countable set $d_c\subseteq c$
cofinal in $c$.  At every other positive successor stage assigned to her,
Even end-extends Odd's preceding condition to
a fresh top height $c$ of cofinality $\omega$, puts
$\ell(c)=\varnothing$, and adds a countable set $d_c\subseteq c$ cofinal in
$c$ with $f(d_c)=-1$.  After every positive move, Even maintains the
following protection invariant at her top height $c$.  The coordinate $c$
belongs to the $\ell$-domain, and either $\cf(c)=\omega$ and the reserved
dummy is present, or $\cf(c)>\omega$.  The prepared successor move is legal.
The new supremum is fresh, the nontrivial support is unchanged, and the dummy
contributes no binary function to the range of $f$.

Let $0<\varepsilon<\alpha$ be a limit stage.  Along any legal play, stage~$2$
is genuine and all later moves lie below it.  For $2\leq\xi<\varepsilon$, write
$p_\xi=\langle t_\xi,\ell_\xi,b_\xi,f_\xi\rangle$ with tree height
$\gamma_\xi+1$.  If stage~$1$ was genuine, $p_2$ extends it, so starting at
stage~$2$ loses no information.  Earlier limit moves make the height sequence continuous,
while the prepared positive Even successor moves ensure that the Even heights form
a strictly increasing cofinal subsequence.  Thus
\[
   \delta=\sup_{2\leq\xi<\varepsilon}\gamma_\xi<\alpha,
   \qquad
   \cf(\delta)=\cf(\varepsilon).
\]
If a cofinal subset of $\delta$ had size below $\cf(\varepsilon)$,
all its witnessing stages would be bounded below $\varepsilon$, so the subset
itself would be bounded in $\delta$, a contradiction.
If $\cf(\varepsilon)=\omega$, apply the $\sigma$-closure construction to a
countable cofinal subsequence, taking the new top height to be $\delta$.  The
Even top heights are cofinal in the union of the $\ell$-domains, so the
construction puts $\ell(\delta)=\varnothing$ and a dummy value at a countable
cofinal subset of $\delta$.  The protection invariant is preserved.

Suppose that $\cf(\varepsilon)>\omega$.  Let
\[
   \widetilde t=\bigcup_{2\leq\xi<\varepsilon}t_\xi,
   \qquad
   B(v)=\bigcup_{\xi_*<\xi<\varepsilon}b_\xi(v),
\]
where $\xi_*\geq2$ is any stage at which $v$ is already present, and put
\[
   t_\varepsilon=\widetilde t\cup
      \{B(v):v\in\widetilde t\}.
\]
Define $b_\varepsilon(v)=B(v)$ on old nodes and let every new top node be its
own $b$-value.  This gives a normal end extension of height $\delta+1$.
The unions of the old $\ell$- and $f$-coordinates are well defined.  The
union of the old $\ell$-domains is closed below $\delta$ by eventual stabilization, and
adding the coordinate $\delta$ closes it.  Every old $f$-value remains
legal because its countable domain is bounded below some earlier top height,
where end extension has already fixed the corresponding restriction set.
Distinct domain suprema and all old $f$--$\ell$ disjointness requirements are
inherited from the descending play.

Let $D$ be the set of top heights occurring at positive Even stages below
$\varepsilon$.  Continuity of the height sequence and the fact that the positive
Even stages are cofinal in $\varepsilon$ make $D$ a club in $\delta$.  No member
of $D$ belongs to the nontrivial support.  For $c\in D$, all $f$-domain
suprema present before $c$ is introduced are below $c$.  At $\cf(c)=\omega$ the
strategy introduces at $c$ only the reserved dummy.  At $\cf(c)>\omega$ it
introduces no $f$-domain there.  The last extension clause in
\cref{def:typed-forcing} requires every $f$-domain added later to have supremum above
$c$.  Consequently, the union of the nontrivial supports is nonstationary in
$\delta$.  At every $\beta<\delta$ of uncountable cofinality, eventual stabilization
ensures that the support is nonstationary in $\beta$, while above $\delta$ the support is bounded.

Let $C\subseteq D$ consist of the protected top heights of cofinality
$\omega$.  The prepared successor moves make $C$ unbounded in $\delta$.
To check countable closure, let $\langle c_n:n<\omega\rangle$ be increasing
in $C$.  If it is not eventually constant, the stages at which the $c_n$
were introduced have a limit supremum below $\varepsilon$, because
$\cf(\varepsilon)>\omega$.  At that limit stage the continuous height
sequence has top height $\sup_n c_n$ of cofinality $\omega$, and the
strategy introduces another protected top.  Thus $\sup_n c_n\in C$.
For an eventually constant sequence, the supremum is one of its terms and
therefore belongs to $C$.
Choose a top node
$x\in(t_\varepsilon)_\delta$ and put
\[
   E_x=\{z\in\PP_{\omega_1}\delta:\sup z\in C\}.
\]
This set is club.  For unboundedness, extend any countable subset of $\delta$
by a countable cofinal subset of a larger member of $C$.  Closure follows from
the $\omega$-closure of $C$ and the positive-length convention for increasing
unions.  Define
\[
   \ell_\varepsilon=
      \bigcup_{2\leq\xi<\varepsilon}\ell_\xi
      \cup\{\langle\delta,
          \{x\}\cup\{x\restriction z:z\in E_x\}\rangle\},
   \qquad
   f_\varepsilon=\bigcup_{2\leq\xi<\varepsilon}f_\xi.
\]
The only new disjointness requirement is at $\delta$.  The top node $x$ is not
an $f$-value because every non-dummy $f$-value has countable domain.  If
$x\restriction z=f_\varepsilon(w)$ for $z\in E_x$, equality of domains gives
$w=z$.  With $c=\sup z\in C$, both $w$ and the reserved dummy domain $d_c$
have supremum $c$.  Uniqueness of domain suprema therefore gives
$w=d_c$.  Hence $f_\varepsilon(w)=-1$, contradicting the assumed equality
$f_\varepsilon(w)=x\restriction z$.  The new ladder coordinate and
all new tree nodes occur strictly above every earlier top height, and the
$b$-values extend all earlier $b$-values.  Thus every order clause holds and
the limit move is a legal condition below the whole play.

This defines Even's strategy through every stage below $\alpha$.  Since
$\alpha$-strategic closure implies $<\alpha$-strategic closure, Cummings's definitions and discussion
\cite[Definition~5.8(3), p.~792, Definition~5.15(1)--(2), p.~793,
and the discussion on p.~794]{Cummings2010} imply that $\Sstar(\alpha)$ adds no
sequence of ordinals of length less than $\alpha$.  Hence it is
$<\alpha$-distributive and preserves $\alpha$ as a regular cardinal.
\end{proof}

By \cref{thm:closure}, the forcing adds no new sequences of ordinals of
length less than $\alpha$.  In particular, it preserves every cardinal
$\leq\alpha$, and $\alpha$ remains regular.

Let $G\subseteq\Sstar(\alpha)$ be generic.  Write
\[
  T_\alpha=\bigcup_{p\in G}t_p,
  \qquad
  \ell_G=\bigcup_{p\in G}\ell_p,
  \qquad
  L_\alpha=\bigcup\rng(\ell_G).
\]
For $u\in T_\alpha$, the designated branch through $u$ is
\[
  B_\alpha(u)=
  \bigcup\{b_p(u):p\in G,\ u\in t_p\}.
\]
Call $L_\alpha$ the generic ladder-coordinate set.  We adapt
Hayut and Magidor's notation
\cite[Notation~5.6, p.~1123]{HayutMagidor2022}, using $L_\alpha$
for this set of nodes and $\ell_G$ for its level-indexing function.
The trace and covering arguments concern this node set rather than a
ladder system in the full two-cardinal sense.
When $T_\alpha$, $L_\alpha$, or $B_\alpha(u)$ for a ground-model node $u$
occurs inside a forcing statement, the same symbol denotes the corresponding
$\Sstar(\alpha)$-name.  We omit check accents on ground-model objects
when no ambiguity arises.

\begin{proposition}\label{prop:level-domain-failure}
Let $\alpha\geq\omega_2$ be regular and let
$G\subseteq\Sstar(\alpha)$ be generic.  Then, with
\[
   D_\alpha=\{\domn(y):y\in L_\alpha\},
\]
no ordinal $\xi\in[\omega_1,\alpha)$ with
$\cf^{V[G]}(\xi)=\omega$ belongs to $D_\alpha$.  In particular,
$D_\alpha$ does not contain a club in $\PP_\alpha\alpha$.  Consequently,
$L_\alpha$ does not satisfy the club-many-level requirement in the definition
of a ladder system on a $\PP_\alpha\alpha$-tree.
\end{proposition}

\begin{proof}
By \cref{thm:closure}, $\alpha$ remains regular and no new countable sequence
of ordinals is added.  If $\omega_1\leq\xi<\alpha$ and
$\cf^{V[G]}(\xi)=\omega$, then $\cf^V(\xi)=\omega$.  If $\xi\in\domn(\ell_G)$, then $\ell_G(\xi)=\varnothing$
by the forcing definition.  Every member of $L_\alpha$ therefore
has either an ordinal domain of uncountable cofinality or a
countable restriction domain.  Hence $\xi\notin D_\alpha$.

Viewed as subsets of $\alpha$, the ordinals in the interval
$[\omega_1,\alpha)$ form a club in $\PP_\alpha\alpha$.  If $D_\alpha$
contained a club, then its intersection with this ordinal club would be
a club subset of $\alpha$.  Since the ordinals of cofinality $\omega$ are
stationary in $\alpha$, that club would contain some
$\xi\geq\omega_1$ of cofinality $\omega$, contradicting
$\xi\notin D_\alpha$.
\end{proof}

This level-domain failure comes from the additional
empty-value convention at coordinates of cofinality at most $\omega$.  It
does not use the unrestricted tree clause.

\begin{lemma}\label{lem:designated-branches}
Let $\alpha$ be an uncountable regular cardinal and let
$G\subseteq\Sstar(\alpha)$ be generic.  Then $T_\alpha$ is a normal binary
tree of height $\alpha$, and $B_\alpha(u)$ is a cofinal branch through
$T_\alpha$ for every $u\in T_\alpha$.
\end{lemma}

\begin{proof}
The union of end extensions is a downward-closed binary tree.  Given a
condition $p$ and $\theta<\alpha$, choose
\[
  \max\{\gamma_p,\theta\}<\gamma<\alpha.
\]
End-extend every old top node by a chain to level $\gamma$, extend each old
$b$-value along the corresponding chain, and assign to every new node its
top continuation.  Leave $\ell_p$ and $f_p$ unchanged.  The resulting
condition has height $\gamma+1$, so conditions of height above $\theta$ are
dense.  When applied below a condition containing a fixed node, the same
construction gives an extension of that node to every prescribed higher
level.  Hence $T_\alpha$ is normal and has height $\alpha$.

For a fixed $u$, the values $b_p(u)$ appearing in the generic filter are
coherent by the ordering.  The height-extension operation can be carried out
below every condition containing $u$, so their union has domain $\alpha$.  Each
of its initial segments belongs to an old tree level and hence to
$T_\alpha$.  Thus the union is a cofinal branch through $T_\alpha$.
\end{proof}

\Needspace{12\baselineskip}
Let $\lambda$ be an ordinal and let $L$ be a set of partial binary functions
on $\lambda$.  A family $\mathcal B\subseteq{}^\lambda2$
\emph{generates} $L$ if every $y\in L$ is extended by some member of
$\mathcal B$.  If $T$ is a tree whose cofinal branches are functions on
$\lambda$, write $\Br(T)$ for the family of those branches and define
\[
   \bcov(T,L)=
   \min\bigl\{|\mathcal B|:\mathcal B\subseteq\Br(T)
          \text{ and }\mathcal B\text{ generates }L\bigr\},
\]
with value $\infty$, understood to exceed every cardinal, when there is no such family.

For comparison with two-cardinal trees, if $\kappa\leq\lambda$ are
cardinals and $\mathcal B\subseteq{}^\lambda2$, put
\[
   L^{\kappa}_{\mathcal B}
      =\{c\restriction x:c\in\mathcal B,\ x\in\PP_\kappa\lambda\}.
\]
Thus $\mathcal B$ generates $L^{\kappa}_{\mathcal B}$ in the preceding
sense.

\Needspace{5\baselineskip}
\begin{lemma}\label{lem:small-cover}
Let $\rho$ be an infinite cardinal, let $\lambda$ be an ordinal, and let $L$ be a set
of partial binary functions on $\lambda$.  Suppose that $L$ is generated by
$\mathcal B\subseteq{}^\lambda2$ with $|\mathcal B|<\rho$.  If
$b\in{}^\lambda2$ meets $L$ $\rho$-cofinally, then $b\in\mathcal B$.
\end{lemma}

\begin{proof}
Assume that $b\notin\mathcal B$.  For each $c\in\mathcal B$, choose
$\beta_c<\lambda$ such that $b(\beta_c)\neq c(\beta_c)$.  The set
\[
   a=\{\beta_c:c\in\mathcal B\}
\]
has cardinality at most $|\mathcal B|<\rho$, and hence belongs to
$\PP_\rho\lambda$.
Choose $z\supseteq a$ with $b\restriction z\in L$.  Since $\mathcal B$
generates $L$, some $c\in\mathcal B$ extends $b\restriction z$.  But
$\beta_c\in z$, contradicting the choice of $\beta_c$.
\end{proof}

\begin{corollary}\label{cor:branch-cover-lower-bound}
Let $\rho$ be an infinite cardinal, let $\lambda$ be an ordinal, let $L$ be a
set of partial binary functions on $\lambda$, and let $T$ be a tree whose
cofinal branches are functions on $\lambda$.  If
$\mathcal C\subseteq\Br(T)$ is a family of pairwise distinct branches of
cardinality at least $\rho$ such that every member of $\mathcal C$ meets $L$
$\rho$-cofinally, then
\[
   \bcov(T,L)\geq\rho.
\]
\end{corollary}

\begin{proof}
Suppose toward a contradiction that $\bcov(T,L)<\rho$.  Then there is a
generating family $\mathcal B\subseteq\Br(T)$ with $|\mathcal B|<\rho$.
By \cref{lem:small-cover}, every member of $\mathcal C$ belongs to
$\mathcal B$, contradicting $|\mathcal C|\geq\rho$.
\end{proof}

The lemma also has a closure interpretation.  Suppose that
$\rho\leq\kappa\leq\lambda$ are infinite cardinals.  Let
$\tau_{<\rho}$ be the topology on ${}^\lambda2$ whose basic neighborhoods
prescribe fewer than $\rho$ coordinates.  For
$\mathcal B\subseteq{}^\lambda2$, a function $b$ meets
$L^{\kappa}_{\mathcal B}$ $\rho$-cofinally if and only if
$b\in\overline{\mathcal B}^{\tau_{<\rho}}$.  Agreement with some member of
$\mathcal B$ on every prescribed small coordinate set is precisely the closure
condition.  Thus \cref{lem:small-cover} is the set-theoretic form of the fact
that a subset of size $<\rho$ is closed in this topology.  We use branch
terminology to reflect the designated cofinal branches supplied by the typed
forcing.  The invariant $\bcov$ measures generation of the whole target set
rather than closure of a single branch.

The strict cardinal inequality in \cref{lem:small-cover} is sharp, and a branch
may catch the generated set without belonging to the chosen generating family.

\begin{proposition}
Let $\rho<\kappa\leq\lambda$ be infinite cardinals.  There exist a
$\PP_\kappa\lambda$-tree $\mathcal T$, a family
$\mathcal B\subseteq\Br(\mathcal T)$ of size $\rho$, a ladder system
$L=L^{\kappa}_{\mathcal B}$, and a branch $b\notin\mathcal B$ such that
\[
   \Tr_\rho(b,L)=\PP_\rho\lambda.
\]
\end{proposition}

\begin{proof}
Regard $\rho$ as a subset of $\lambda$.  Let $b$ be constantly zero.  For
$\xi<\rho$, let $c_\xi$ agree with $b$ except at coordinate $\xi$, and put
$\mathcal B=\{c_\xi:\xi<\rho\}$.  Define
\[
   \mathcal T_x=
   \{b\restriction x\}\cup
   \{c_\xi\restriction x:\xi<\rho\}
   \qquad(x\in\PP_\kappa\lambda).
\]
Every level is nonempty, the system is closed under restrictions, and each
level has size at most $\rho<\kappa$.  Hence $\mathcal T$ is a
$\PP_\kappa\lambda$-tree.  By construction, $b$ and every $c_\xi$ are
branches through $\mathcal T$.  The family $L^{\kappa}_{\mathcal B}$ is a
ladder system.  It meets every level.  If
$\eta=c_\xi\restriction x\in L^{\kappa}_{\mathcal B}$ and
$\cf(|x\cap\kappa|)>\omega$, take
$E_\eta=\PP_{|x\cap\kappa|}x$.  Then $E_\eta$ is club and
$\eta\restriction y=c_\xi\restriction y\in L^{\kappa}_{\mathcal B}$
for every $y\in E_\eta$.

Fix $x\in\PP_\rho\lambda$.  Since $|x|<\rho$, choose
$\xi\in\rho\setminus x$.  Then
\[
   b\restriction x=c_\xi\restriction x\in L^{\kappa}_{\mathcal B}.
\]
Thus the trace is all of $\PP_\rho\lambda$, although
$b\notin\mathcal B$.
\end{proof}

The example shows that the strict inequality $|\mathcal B|<\rho$
in \cref{lem:small-cover} cannot be relaxed.  Under the hypotheses of
\cref{cor:branch-cover-lower-bound}, generating families of size less
than $\rho$ are excluded, while families of size $\rho$ are not ruled
out.  For $\rho=\omega_1$, this excludes countable generating families
without deciding whether one of size $\omega_1$ exists.

\begin{proposition}\label{prop:designated-cover}
Let $\alpha$ be an uncountable regular cardinal and let
$G\subseteq\Sstar(\alpha)$ be generic.  Then the family of cofinal
branches
\[
   \mathcal B_\alpha=\{B_\alpha(u):u\in T_\alpha\}
\]
generates $L_\alpha$.
\end{proposition}

\begin{proof}
By \cref{lem:designated-branches}, the displayed functions are cofinal
branches through $T_\alpha$.  Fix $y\in L_\alpha$, and choose $p\in G$ and
$\xi\in\domn(\ell_p)$ with
$y\in\ell_p(\xi)$.  Values at coordinates of cofinality at most
$\omega$ are empty.  At a coordinate of uncountable cofinality, the definition
of $\Sstar(\alpha)$ gives
either $y=x$ or $y=x\restriction z$ for some
$x\in(t_p)_\xi$.  In both cases the designated branch $B_\alpha(x)$
extends $y$.  Hence $\mathcal B_\alpha$ generates $L_\alpha$.
\end{proof}

\section{Bistationary Traces, Wide Levels, and Branch-Cover Rigidity}

\subsection{Fixed-seed ladder insertion}\label{sec:insertion}

The $b$-coordinate in the Hayut--Magidor construction
\cite[Definition~5.5(3), p.~1123]{HayutMagidor2022}
supplies coherent designated extensions.  In the typed forcing,
every $z\in\domn(f_p)$ is a countable subset of $\gamma_p$.
For any $\delta>\gamma_p$, the club
$\{w\in\PP_{\omega_1}\delta:\gamma_p\in w\}$ is therefore disjoint from
$\domn(f_p)$, so the fixed-seed insertion needs no
$\omega_1$-stage recursion.

\begin{lemma}\label{lem:targeted-insertion}
Let $\alpha\geq\omega_2$ be regular,
$p\in\Sstar(\alpha)$, $u\in t_p$, and $\theta<\alpha$.  There exist
$q\leq p$, an ordinal $\theta<\delta<\alpha$ with
$\cf(\delta)=\omega_1$, a node
\[
   x=b_q(u)\in(t_q)_\delta
\]
extending $u$, and a club $E_x\subseteq\PP_{\omega_1}\delta$ such that
\[
  \delta\in\domn(\ell_q)
  \quad\text{and}\quad
  \ell_q(\delta)=\{x\}\cup
       \{x\restriction z:z\in E_x\}.
\]
\end{lemma}

\begin{proof}
Choose $\delta<\alpha$ of cofinality $\omega_1$ above
$\max\{\theta,\gamma_p\}$.  For each $s\in(t_p)_{\gamma_p}$ choose
an extension $s^\star\in{}^\delta2$, taking, for example,
the continuation by zeroes.  Define
\[
  t_q=t_p\cup
      \{s^\star\restriction\xi:
                s\in(t_p)_{\gamma_p},\ \gamma_p<\xi\leq\delta\}.
\]
This is a normal end extension of $t_p$.  For $v\in t_p$, let
$b_q(v)=b_p(v)^\star$.  On a new node lying above the old top node
$s$, let $b_q$ take the value $s^\star$.  In particular,
\[
  x=b_q(u)=b_p(u)^\star\in(t_q)_\delta.
\]

Now set
\[
  E_x=\{z\in\PP_{\omega_1}\delta:\gamma_p\in z\},
  \qquad f_q=f_p,
\]
and extend $\ell_p$ by the single new coordinate
\[
  \ell_q(\delta)=\{x\}\cup\{x\restriction z:z\in E_x\}.
\]
The set $E_x$ is club, since it is unbounded under inclusion
and closed under nonempty increasing countable unions.
The domain of $\ell_q$ remains closed, and every old disjointness
requirement is inherited from $p$.

Every non-dummy $f_p(w)$ has domain $w\subseteq\gamma_p$,
whereas every $x\restriction z$ with $z\in E_x$ has
$\gamma_p\in z$.  Moreover, the top node $x$ has uncountable
domain $\delta$, while every non-dummy $f_p$-value has countable
domain.  Hence $\rng(f_q)\cap\ell_q(\delta)=\varnothing$.
No $f$-value was added, so the support, distinct-supremum, and
freshness clauses are unchanged.  Thus $q\in\Sstar(\alpha)$,
$q\leq p$, and $q$ has the asserted properties.
\end{proof}

The fixed seed is carried by the extended $b$-coordinate.
This direct end extension needs no limit-stage choice of a
coherent branch union.

An anchored refinement places one prescribed countable set among the
restrictions at the new ladder level.

\begin{lemma}\label{lem:anchored-insertion}
Let $\alpha\geq\omega_2$ be regular, let $p\in\Sstar(\alpha)$, and let
$u\in t_p$.  Suppose that $a\in\domn(f_p)$ and $f_p(a)=-1$.  There exist
$q\leq p$, an ordinal $\delta<\alpha$ of cofinality $\omega_1$, a node
$x=b_q(u)\in(t_q)_\delta$, and a club
$E_x\subseteq\PP_{\omega_1}\delta$ such that
\[
  a\in E_x
  \quad\text{and}\quad
  \ell_q(\delta)=\{x\}\cup
       \{x\restriction z:z\in E_x\}.
\]
\end{lemma}

\begin{proof}
Apply \cref{lem:targeted-insertion} with $\theta=\gamma_p$, obtaining
$q^0\leq p$, $\delta$, $x$, and $E_x^0$.  Then $f_{q^0}(a)=-1$ and
$a\subseteq\gamma_p<\delta$, so $a\in\PP_{\omega_1}\delta$.  Put
$E_x=E_x^0\cup\{a\}$.  Then $E_x$ is still club.  If an increasing
sequence has successor length, its union is its last member.  If it has limit
length, either cofinally many terms belong to $E_x^0$ or the sequence is
eventually constant at $a$.

Let $q$ agree with $q^0$ except that
$\ell_q(\delta)=\{x\}\cup\{x\restriction z:z\in E_x\}$.  The only possible new
node is $x\restriction a$.  If $x\restriction a=f_q(w)$, equality of domains
gives $w=a$, contradicting $f_q(a)=-1$.  Hence $q$ is a condition, and
$\delta>\gamma_p$ gives $q\leq p$.
\end{proof}

\subsection{Bistationary traces}\label{sec:traces}

We can force a ground-model countable set $a$ to belong to a named club
while reserving $\sup a$ by assigning the dummy value $f(a)=-1$.

\begin{lemma}\label{lem:club-point}
Let $\alpha$ be an uncountable regular cardinal, let $p\in\Sstar(\alpha)$, and suppose
that $p\forces\dot C$ is club in $\PP_{\omega_1}\alpha$.  There exist
$s\leq p$ and $a\in\PP_{\omega_1}\alpha$ such that, with
$\delta=\sup a$,
\[
   s\forces \check a\in\dot C,
   \qquad
   \ell_s(\delta)=\varnothing,
   \qquad
   f_s(a)=-1.
\]
The top height of $t_s$ may be taken to be $\delta$.
\end{lemma}

\begin{proof}
Choose a sufficiently large regular cardinal $\chi$ and a countable
$M\prec H(\chi)$ containing $p,\dot C,\Sstar(\alpha)$, and $\alpha$.  Put $a=M\cap\alpha$ and
$\delta=\sup a$.  The set $a$ has no
maximum.  If $\beta\in M\cap\alpha$, then
$\beta+1\in M\cap\alpha$ by elementarity.  Hence
$a\subseteq\delta$ and $\cf(\delta)=\omega$.  Enumerate
$a=\{\eta_n:n<\omega\}$.  The ordinals of cofinality $\omega$
are cofinal in $M\cap\alpha$.  For every
$\beta\in M\cap\alpha$, elementarity gives
$\beta+\omega\in M\cap\alpha$, and this ordinal has cofinality
$\omega$.  This supplies the fresh heights used below.

By \cref{thm:closure}, every name forced to belong to
$\PP_{\omega_1}\alpha$ can be decided as a ground-model countable set.  First
strengthen the condition to decide whether the name is empty.  In the nonempty
case, use the maximum principle to choose a name for a surjection from $\omega$
onto the named set, decide its values coordinate by coordinate, and take a
common lower bound.  Using elementarity at each finite stage, construct a
descending sequence $\langle p_n:n<\omega\rangle$ with $p_n\in M$ and
$p_0=p$, and an increasing sequence $\langle y_n:n<\omega\rangle$ in
$M\cap\PP_{\omega_1}\alpha$.  Choose $q_n\in M$ and $y_n$ such that
$q_n\leq p_n$ and
\[
  q_n\forces \check y_n\in\dot C
  \quad\text{and}\quad
  y_{n-1}\cup\{\eta_0,\ldots,\eta_n\}\subseteq y_n,
\]
where $y_{-1}=\varnothing$.  End-extend $q_n$ to a condition $p_{n+1}$ with
fresh top height $\gamma_n\in M\cap\alpha$ of cofinality $\omega$, chosen above
$\eta_0,\ldots,\eta_n$ and, when $n>0$, above $\gamma_{n-1}$.  Then
$\langle\gamma_n:n<\omega\rangle$ is increasing and cofinal in $\delta$.
Put
$\ell_{p_{n+1}}(\gamma_n)=\varnothing$, choose
$z_n\in M\cap\PP_{\omega_1}\gamma_n$ cofinal in $\gamma_n$, and set
$f_{p_{n+1}}(z_n)=-1$.  Extend every old $b$-value to the new top.  Freshness
of $\gamma_n$ preserves the uniqueness of domain suprema and all old
$f$--$\ell$ clauses.

Since $y_n\in M$ is countable, $y_n\subseteq M\cap\alpha=a$.  The displayed
inclusions give $\bigcup_n y_n=a$.  Use the countable-limit construction from
\cref{thm:closure} with top $\delta$, $\ell_s(\delta)=\varnothing$,
and dummy $f_s(a)=-1$.  No earlier member of $\domn(f_s)$ has supremum
$\delta$.

The condition $s$ forces $y_n\in\dot C$ for every $n$.  Since $\dot C$ is
closed under increasing $\omega$-unions and $a=\bigcup_{n<\omega}y_n$ in the
ground model,
\[
   s\forces \check a\in\dot C.
\]
\end{proof}

The club point supplied by the lemma can be combined with
$<\alpha$-distributivity to handle fewer than $\alpha$ designated branches.

\Needspace{12\baselineskip}
\begin{theorem}\label{thm:stationary-traces}
Let $\alpha\geq\omega_2$ be regular and let
$G\subseteq\Sstar(\alpha)$ be generic.  For every nonempty
$U\subseteq T_\alpha$ with $|U|<\alpha$,
\[
   \bigcap_{u\in U}
   \Tr_{\omega_1}(B_\alpha(u),L_\alpha)
\]
is stationary in $\PP_{\omega_1}\alpha$.
\end{theorem}

\begin{proof}
Fix a cardinal $0<\lambda<\alpha$.  By the maximum principle, it is enough
to consider names $\langle\dot u_\xi:\xi<\lambda\rangle$ for nodes of
$T_\alpha$ and a name $\dot C$ for a club in $\PP_{\omega_1}\alpha$.
Let $p$ force these properties, and fix $r\leq p$.

The $<\alpha$-distributivity in \cref{thm:closure} allows a single
strengthening to decide the entire sequence of nodes.  To see this, first
decide the $\lambda$-sequence of their ordinal lengths.  Regularity of
$\alpha$ bounds all the decided lengths below one ordinal less than
$\alpha$.  The binary values of all $\lambda$ nodes can then be encoded
by a sequence of fewer than $\alpha$ ordinals and decided simultaneously.
Raise the tree height above all their lengths.  Since the resulting
condition still forces each decided node to belong to the generic tree,
end extension ensures that every such node belongs to its condition tree.
We may therefore fix $r^*\leq r$ and nodes $u_\xi\in t_{r^*}$ such that
\[
   r^*\forces\dot u_\xi=\check u_\xi
   \qquad(\xi<\lambda).
\]

Apply \cref{lem:club-point} below $r^*$ to obtain $s_0\leq r^*$ and
$a\in\PP_{\omega_1}\alpha$ satisfying
\[
   s_0\forces\check a\in\dot C,
   \qquad f_{s_0}(a)=-1.
\]
For each $\xi<\lambda$, let $D_\xi$ be the set of conditions
$s\leq s_0$ for which some $\beta\in\domn(\ell_s)$ satisfies
\[
   b_s(u_\xi)\restriction a\in\ell_s(\beta).
\]
Each $D_\xi$ is dense below $s_0$ by
\cref{lem:anchored-insertion}, because every stronger condition retains
$f(a)=-1$.  It is also open downward, since both the ladder-coordinate
value and the restriction of the designated $b$-value to $a$ are preserved
by end extension.  The forcing below $s_0$ is
$<\alpha$-distributive.  Consequently, the intersection of the
$\lambda<\alpha$ dense open sets $D_\xi$ is dense below $s_0$.
Choose $q$ in this intersection.  It forces
\[
   \check a\in\dot C
   \quad\text{and}\quad
   B_\alpha(u_\xi)\restriction a\in L_\alpha
   \qquad(\xi<\lambda).
\]
Since $r\leq p$ was arbitrary, the common trace meets every named club.
\end{proof}

The proof gives the following dense formulation.  Let
$\alpha\geq\omega_2$ be regular, let $p\in\Sstar(\alpha)$, and let
$\langle u_\xi:\xi<\lambda\rangle$ be a sequence in $t_p$ of length
$\lambda<\alpha$.  Suppose that $p$
forces $\dot C$ to be club in $\PP_{\omega_1}\alpha$.  For every $r\leq p$
there exist
$q\leq r$ and $a\in\PP_{\omega_1}\alpha$ in the ground model such that
\[
  q\forces \check a\in\dot C
  \quad\text{and}\quad
  q\forces B_\alpha(u_\xi)\restriction a\in L_\alpha
  \qquad(\xi<\lambda).
\]
The same reserved domain $a$ persists throughout the dense-open argument, so
$q$ forces $a$ to belong to every designated trace in the sequence.  No cardinal-arithmetic or large-cardinal hypothesis is used
beyond the stated assumptions on $\alpha$.

The next proposition is a typed analogue of the no-club-catching result
of Hayut and Magidor \cite[Claim~5.8, p.~1124]{HayutMagidor2022}.
Its proof uses a non-dummy value in place of the dummy value in the
countable-limit construction.

\begin{proposition}\label{prop:no-club}
Let $\alpha\geq\omega_2$ be regular and let
$G\subseteq\Sstar(\alpha)$ be generic.  No cofinal branch through
$T_\alpha$ meets $L_\alpha$ on an $\omega_1$-club.
\end{proposition}

\begin{proof}
Let $p$ force that $\dot c$ is a cofinal branch through $T_\alpha$ and that
$\dot C$ is a club in $\PP_{\omega_1}\alpha$.  Fix $r\leq p$ and choose a
sufficiently large regular cardinal $\chi$ and a countable
$M\prec H(\chi)$ containing
$r,\dot c,\dot C,\Sstar(\alpha)$, and $\alpha$.  Put
$a=M\cap\alpha$ and $\delta=\sup a$.  If $\beta\in a$, then
$\beta+1\in a$ by elementarity, so $a$ has no maximum.  Hence
$a\subseteq\delta$, $\cf(\delta)=\omega$, and $\widehat a=\delta$.
Enumerate $a=\{\eta_n:n<\omega\}$.  As above, the ordinals
of cofinality $\omega$ are cofinal in $M\cap\alpha$, since
$\beta+\omega\in M\cap\alpha$ whenever $\beta\in M\cap\alpha$.

Using elementarity at each finite stage, construct a descending sequence
$\langle p_n:n<\omega\rangle$ with $p_n\in M$ and $p_0=r$, an increasing sequence
$\langle y_n:n<\omega\rangle$ in
$M\cap\PP_{\omega_1}\alpha$, a sequence of ordinals
$\langle\gamma_n:n<\omega\rangle$, and nodes
$v_n\in(t_{p_{n+1}})_{\gamma_n}$.  Given $p_n$, first choose
$r_n\leq p_n$ and $y_n$ such that
\[
  r_n\forces \check y_n\in\dot C,
  \qquad
  y_{n-1}\cup\{\eta_0,\ldots,\eta_n\}\subseteq y_n,
\]
where $y_{-1}=\varnothing$.  End-extend $r_n$ to a condition $s_n$ with a fresh top height
$\gamma_n$ of cofinality $\omega$, chosen above
$\eta_0,\ldots,\eta_n$ and, when $n>0$, above $\gamma_{n-1}$.  Put an empty
ladder-coordinate value there and a dummy $f$-value at a cofinal subset of $\gamma_n$.
The resulting sequence of heights is increasing and cofinal in $\delta$.
Using the $<\alpha$-distributivity from \cref{thm:closure}, strengthen
$s_n$ to $p_{n+1}$ and decide
\[
  p_{n+1}\forces
  \dot c\restriction\gamma_n=\check v_n
\]
for some $v_n\in(t_{p_{n+1}})_{\gamma_n}$.  By elementarity, these choices
can be made in $M$.  The nodes $v_n$ are coherent because the conditions are
descending and the heights are increasing.

Let $y=\bigcup_n v_n\in{}^\delta2$ and $x=y\restriction a$.  Use the
tree, ladder, and $b$-coordinate part of the countable-limit construction from
\cref{thm:closure}, adding $y$ to the new top level alongside the
coherent $b$-limits.  Every initial segment of $y$ already belongs
to one of the earlier trees, so the enlarged tree remains normal.  Retain the
$b$-values of that construction and set $b_q(y)=y$.  Put
$\ell_q(\delta)=\varnothing$.  At the new supremum, use the non-dummy value
\[
   f_q(a)=x
\]
in place of the dummy value used in that construction.
This value has the required type, since $\widehat a=\delta$ and
$x=y\restriction a\in t_q[a]$.  No earlier $f$-domain has supremum
$\delta$.  End extension stabilizes the old nontrivial support below each
bounded ordinal, and adjoining the single point $\delta$ therefore preserves
nowhere stationarity.  The new value is not a member of any old ladder-coordinate value because its
domain $a$ is cofinal in $\delta$, whereas every node occurring in an old
ladder-coordinate value has domain bounded below $\delta$.  Thus $q$ is a condition below every $p_n$.

Since each $y_n\in M\cap\PP_{\omega_1}\alpha$ is countable,
$y_n\subseteq M\cap\alpha=a$.  The displayed inclusions therefore give
$a=\bigcup_n y_n$.  Closure of $\dot C$ under increasing $\omega$-unions
then gives $q\forces\check a\in\dot C$.  Moreover, for every $\beta\in a$, choose $n$
with $\beta<\gamma_n$.  Since $q\leq p_{n+1}$ and
$p_{n+1}\forces\dot c\restriction\gamma_n=\check v_n$, while
$y\restriction\gamma_n=v_n$, the condition $q$ forces
$\dot c\restriction a=x$.  Since $x\in\rng(f_q)$ and every stronger
condition must keep its ladder-coordinate values disjoint from the range of its
$f$-coordinate, the condition $q$ forces $x\notin L_\alpha$.  Hence the trace of $\dot c$ misses
the member $a$ of the named club.  Since $r\leq p$ was arbitrary, no branch
meets $L_\alpha$ on a club.
\end{proof}

\begin{corollary}\label{cor:bistationary}
Let $\alpha\geq\omega_2$ be regular and let
$G\subseteq\Sstar(\alpha)$ be generic.  Every nonempty subfamily of
\[
   \bigl\{\Tr_{\omega_1}(B_\alpha(u),L_\alpha):u\in T_\alpha\bigr\}
\]
of cardinality less than $\alpha$ has bistationary intersection.
In particular, every designated branch meets
$L_\alpha$ $\omega_1$-cofinally, but no designated branch meets it on an
$\omega_1$-club.
\end{corollary}

\begin{proof}
Stationarity of every nonempty intersection of cardinality less than $\alpha$ follows from
\cref{thm:stationary-traces}.  Fix such a subfamily and one of its members.
By \cref{prop:no-club}, that trace contains no club.  If the complement of
the intersection were nonstationary, some club would be contained in the
intersection and hence in the chosen member, a contradiction.  Thus the
intersection is also costationary.  Every stationary subset of
$\PP_{\omega_1}\alpha$ is cofinal under inclusion.
\end{proof}

\subsection{Wide levels and branch-cover rigidity}\label{sec:cover-rigidity}

The designated branch cover supplied by \cref{prop:designated-cover} can
be reduced to $\alpha$ branches, but cannot be reduced to fewer than
$\alpha$ branches in the unrestricted forcing.

\begin{lemma}\label{lem:wide-level}
Let $\alpha$ be an uncountable regular cardinal, let $\mu<\alpha$ be an
infinite cardinal in $V$, and let $p\in\Sstar(\alpha)$.  There are
$q\leq p$, a limit ordinal $\delta<\alpha$, and a node
$u\in(t_p)_{\gamma_p}$ such that
\[
   \{u^\frown d:d\in{}^\mu2\}\subseteq(t_q)_\delta.
\]
Consequently, in every $\Sstar(\alpha)$-generic extension, some level of
$T_\alpha$ contains a copy of $({}^\mu2)^V$.
\end{lemma}

\begin{proof}
Put $\eta=\gamma_p$ and fix $u\in(t_p)_\eta$.  Above $u$, attach a copy of
${}^{<\mu}2$.  For $s\in{}^\xi2$ with $\xi<\mu$, the corresponding node
$u^\frown s$ lies on level $\eta+\xi$.  At the limit level
\[
   \delta=\eta+\mu
\]
add every node $u^\frown d$ with $d\in{}^\mu2$.  Above every other old top
node, add a single chain to level $\delta$.  Since $\alpha$ is a cardinal and
$\eta,\mu<\alpha$, the ordinal $\delta+1$ is below $\alpha$.

Every node in the binary subtree above $u$ has a continuation on the new top
level, and every node on one of the other chains has its chain continuation.
Extend each old $b$-value to a compatible new top node and assign a top
continuation to every new node.  The resulting tree is normal and end-extends
$t_p$, while the new $b$-coordinate coherently extends $b_p$.

Leave $\ell_p$ and $f_p$ unchanged.  Since
$\domn(\ell_p)\subseteq\eta+1$ and
$\domn(f_p)\subseteq\PP_{\omega_1}\eta$, every limit point of
$\domn(\ell_p)$ in $\delta+1$ is at most $\eta$ and belongs to the
old closed domain.  Moreover, for every
$z\in\domn(f_p)$, $\widehat z\leq\eta$, and hence the end extension gives
$t_q[z]=t_p[z]$.  Thus all old non-dummy values remain legal.  Below or at
$\eta$, nowhere stationarity is inherited from $p$.  Above $\eta$, the old
support is bounded.  Every old $f$--$\ell$ disjointness clause is also
preserved.  The resulting tuple is a
condition $q\leq p$ with the displayed top-level copy of ${}^\mu2$.

For fixed $\mu$, the conditions supplied by the first part form a dense set.
A generic filter meets it, and end extension preserves the displayed level.
\end{proof}

\begin{corollary}\label{cor:hm-restriction-wide}
Let $\alpha\geq\omega_2$ be regular and let
$G\subseteq\Sstar(\alpha)$ be generic.  For
$x\in\PP_{\omega_2}\alpha$, put
\[
   \mathcal T^{\omega_2}_{\alpha,x}
      =\{r\restriction x:r\in(T_\alpha)_{\sup(x)+1}\}.
\]
Then some level of this restriction family has size at least $\omega_2$.
Consequently, it is not a $\PP_{\omega_2}\alpha$-tree.
\end{corollary}

\begin{proof}
Use the construction from \cref{lem:wide-level} with
$\mu=\omega_1$.  If $\eta$ is the old top height, that construction takes
$\delta=\eta+\omega_1$.  Put
\[
   x=[\eta,\delta)=\{\eta+\xi:\xi<\omega_1\}.
\]
Then $x\in\PP_{\omega_2}\alpha$.  The nodes
$u^\frown d$ for $d\in({}^{\omega_1}2)^V$ lie on level $\delta$ and have
pairwise distinct restrictions to $x$.  By normality, extend them to level
$\delta+1$.  Their restrictions to $x$ are unchanged.  Hence there is an
injection
\[
   ({}^{\omega_1}2)^V\longrightarrow
      \mathcal T^{\omega_2}_{\alpha,x}.
\]
In $V$, fix an injection
$\omega_2\longrightarrow({}^{\omega_1}2)^V$.  By \cref{thm:closure},
$\omega_2$ is preserved, so in $V[G]$ the composition of these injections
witnesses that the displayed level has size at least $\omega_2$.
\end{proof}

The construction in \cref{lem:wide-level} modifies only $t$ and $b$,
leaving $\ell$ and $f$ unchanged and fixing all earlier tree levels.  It
is legal under the normality convention fixed in
Subsection~\ref{sec:preliminaries}.  For
$\mu=\omega_1$, let $\eta$ be the old top height and put
$\delta=\eta+\omega_1$ and $x=[\eta,\delta)$.  For the condition $q$ constructed there, the ground model $V$
contains an injection
\[
  ({}^{\omega_1}2)^V\longrightarrow
  \{r\restriction x:r\in(t_q)_\delta\}.
\]
Since $\sup(x)=\delta$, the restriction level on the right is indexed
by $\sup(x)$, as in the formula of
\cite[p.~1122]{HayutMagidor2022}.  The endpoint-corrected generic
restriction system in \cref{cor:hm-restriction-wide} instead uses
level $\delta+1$, as verified in its proof.  This calculation does not establish
that $q$ is a condition in the printed forcing of
\cite[Definition~5.5, p.~1123]{HayutMagidor2022}.

\begin{theorem}\label{thm:exact-cover}
Let $\alpha\geq\omega_2$ be regular.  In every
$\Sstar(\alpha)$-generic extension, the ordinal-tree branch-covering
number of the ladder-coordinate set is exactly
\[
    \bcov(T_\alpha,L_\alpha)=\alpha.
\]
\end{theorem}

\begin{proof}
The lower bound follows from the uniform diagonal in
\cref{prop:uniform-diagonal}, which excludes every generating family
of cardinality less than $\alpha$.

For the upper bound, consider the generic ladder-coordinate function
$\ell_G$.  For each $\xi\in\domn(\ell_G)$ of cofinality at most
$\omega$, the empty-value convention gives
$\ell_G(\xi)=\varnothing$.  For each remaining
$\xi\in\domn(\ell_G)$, choose a node $x_\xi\in(T_\alpha)_\xi$
for which the corresponding ladder-coordinate value has the form
\[
   \ell_G(\xi)=\{x_\xi\}\cup
      \{x_\xi\restriction z:z\in E_\xi\}
\]
for a club $E_\xi\subseteq\PP_{\omega_1}\xi$.  The designated branch
$B_\alpha(x_\xi)$ extends every member of $\ell_G(\xi)$.
There are at most $\alpha$ such coordinates, so the family
\[
   \{B_\alpha(x_\xi):
       \xi\in\domn(\ell_G),\ \cf(\xi)>\omega\}
\]
generates $L_\alpha$ and has cardinality at most $\alpha$.
This proves the reverse inequality.
\end{proof}

Write $\Sstar_{\mathrm K}(\alpha)$ for the subposet of
$\Sstar(\alpha)$ comprising all conditions with
\[
  |(t_p)_\xi|\leq|\xi|+\omega
  \qquad(\xi\leq\gamma_p),
\]
equipped with the inherited order.  We refer to it as the
Kurepa-bounded companion.

\begin{lemma}\label{lem:kurepa-split}
Let $\alpha\geq\omega_2$ be regular and
$p\in\Sstar_{\mathrm K}(\alpha)$.  For every $\theta<\alpha$,
there are $q\leq p$, ordinals $\theta<\xi<\alpha$, and two nodes
$v_0,v_1\in(t_q)_{\xi+1}$ such that
\[
  v_0=b_q(\varnothing),\qquad
  v_0\restriction\xi=v_1\restriction\xi,
  \qquad v_0(\xi)\ne v_1(\xi).
\]
In particular, $q$ forces the designated branches through $v_1$ and the
root to agree below $\xi$ and disagree at $\xi>\theta$.
\end{lemma}

\begin{proof}
Choose $\xi>\max\{\theta,\gamma_p,\omega\}$ with
$\xi+2<\alpha$.  End-extend $p$ to a condition $p^*$ of top
height $\xi$ by giving each old top node just one continuation on
each intervening level.  Retain the old $\ell$- and $f$-coordinates.
This preserves the Kurepa level bound because the number of nodes
on each new level is at most
$|(t_p)_{\gamma_p}|\leq|\gamma_p|+\omega\leq|\zeta|+\omega$
for every $\gamma_p<\zeta\leq\xi$.

Put $s=b_{p^*}(\varnothing)\in(t_{p^*})_\xi$.
To reach level $\xi+1$, extend $s$ by both successors
$s^\frown0$ and $s^\frown1$, while giving every other node on
level $\xi$ only its zero successor.  For every $v\in t_{p^*}$, put
$b_q(v)=b_{p^*}(v)^\frown0$.
For every new top node $v$, put $b_q(v)=v$.  In particular,
$b_q(\varnothing)=s^\frown0$.  Leave $\ell$ and $f$
unchanged.  The new top level has cardinality at most
$|(t_{p^*})_\xi|+1\leq|\xi+1|+\omega$.
All old restriction levels, disjointness requirements, and support
conditions remain unchanged by end extension.  Thus the resulting
condition $q$ belongs to $\Sstar_{\mathrm K}(\alpha)$ and extends
$p$.  The nodes $v_0=s^\frown0$ and $v_1=s^\frown1$ have the
required properties.
\end{proof}

\begin{proposition}\label{prop:kurepa-bounded-persistence}
Let $\alpha\geq\omega_2$ be regular.  The Kurepa-bounded forcing $\Sstar_{\mathrm K}(\alpha)$ defined above is
$\sigma$-closed.  After adjoining a formal maximum, it is
$\alpha$-strategically closed.  Write
$T^{\mathrm K}_\alpha$, $L^{\mathrm K}_\alpha$, and
$B^{\mathrm K}_\alpha(u)$ for its generic tree, ladder-coordinate set, and
designated branches.  Then, for every nonempty
$U\subseteq T^{\mathrm K}_\alpha$ with $|U|<\alpha$, the intersection
\[
   \bigcap_{u\in U}
   \Tr_{\omega_1}(B^{\mathrm K}_\alpha(u),L^{\mathrm K}_\alpha)
\]
is bistationary.  Moreover, the ordinal-tree branch-covering number
satisfies
\[
  \bcov(T^{\mathrm K}_\alpha,L^{\mathrm K}_\alpha)=\alpha.
\]
\end{proposition}

\begin{proof}
We verify that the closure, fixed-seed insertion, club-point, and
no-club constructions used here preserve the Kurepa level bound.
Every genuine Odd move satisfies the bound by definition.  In the
successor height extensions used in these constructions, each old top
node continues along a single chain, so no new level exceeds the old
top level in size.
At a nonzero limit height $\delta$, the constructions add at most one coherent
top limit for each node already present below $\delta$.  If $t$ is the union of
the old trees below $\delta$, then
\[
   |t|
      \leq \sum_{\xi<\delta}(|\xi|+\omega)
      \leq |\delta|+\omega.
\]
Adding the single extra top node in the no-club argument does not
change the bound.  The basic height extension used to obtain the generic
tree and designated branches also preserves the bound.  Thus the
countable lower-bound construction proving $\sigma$-closure, Even's
strategy, and the insertion, club-point, and no-club arguments all remain
inside $\Sstar_{\mathrm K}(\alpha)$.  As in
\cref{thm:closure}, strategic closure gives $<\alpha$-distributivity and
preserves $\omega_1$ and $\alpha$.  The stationary-trace and no-club arguments
therefore give the stated bistationarity conclusion for the companion forcing.

The full width conclusion of \cref{lem:wide-level} is unnecessary for the
branch-cover argument.  Given $p\in\Sstar_{\mathrm K}(\alpha)$, take a
chain-only end extension $p^*\leq p$ of top height
$\eta>\gamma_p$ such that $|\eta|\geq\omega_1$ and
$\ell_{p^*}=\ell_p$, $f_{p^*}=f_p$.  Regularity of $\alpha$ gives
$\eta+\omega<\alpha$.  Choose $D\subseteq{}^\omega2$ of size
$\omega_1$.  Above one node $u\in(t_{p^*})_\eta$,
place the restrictions
\[
   \{u^\frown(d\restriction n):d\in D\}
\]
on level $\eta+n$ for $n<\omega$, and place
$\{u^\frown d:d\in D\}$ on level $\eta+\omega$.  Above each other
node of $(t_{p^*})_\eta$, add a single chain.  The split part contributes
at most $\omega_1$ nodes
to any new level and the remaining chains at most $|\eta|+\omega$, by the bound
at level $\eta$.  Hence every new level has size at most $|\eta|+\omega$.
Extend every $b_{p^*}$-value to a compatible new top node.  For each new
split node below level $\eta+\omega$, choose a member of $D$ extending
its finite tail and use the corresponding top node as its $b$-value.
For each new intermediate node on a remaining chain, use that chain's
top node as its $b$-value.  Let each new top node be its own $b$-value.
Keep the $\ell$- and $f$-coordinates unchanged from $p^*$.
Since $\domn(\ell_{p^*})=\domn(\ell_p)\subseteq\gamma_p+1$ and
$z\subseteq\gamma_p$ for each $z\in\domn(f_{p^*})$, with
$\gamma_p<\eta$, end extension fixes the restriction levels relevant to
all old non-dummy $f$-values and preserves the old
$f$--$\ell$ disjointness clauses.  The resulting condition extends
$p^*$ and hence $p$.

Hence conditions whose trees have a level with at least $\omega_1$
distinct nodes are dense in $\Sstar_{\mathrm K}(\alpha)$.  Such a level gives $\omega_1$ pairwise distinct designated branches.  By the
bistationarity conclusion, their traces are stationary and hence
$\omega_1$-cofinal.  Applying
\cref{cor:branch-cover-lower-bound} with $\rho=\omega_1$ rules out a
countable generating family.
To obtain the exact lower bound, work in the generic extension and
fix the designated branch $B^{\mathrm K}_\alpha(\varnothing)$.
For every $\theta<\alpha$, \cref{lem:kurepa-split} shows that
conditions with a splitting node above $\theta$ on the
designated root branch are dense.  Hence, in the generic
extension, we may choose a splitting coordinate $\xi>\theta$ and a
side node $v_\xi$ whose designated branch agrees with
$B^{\mathrm K}_\alpha(\varnothing)$ below $\xi$ and disagrees
with it at $\xi$.  Thus the set of such splitting coordinates is
unbounded in $\alpha$ and has cardinality $\alpha$.  The
corresponding designated side branches are pairwise distinct.
If $\xi<\zeta$, the branch selected at $\xi$ disagrees
with the root branch there, while the branch selected at $\zeta$
still agrees with the root branch at $\xi$.

Every designated branch $B^{\mathrm K}_\alpha(v)$ meets
$L^{\mathrm K}_\alpha$ $\alpha$-cofinally.  Fix
$x\in\PP_\alpha\alpha$ in the generic extension.
Since the bounded forcing preserves the regularity of $\alpha$,
we have $\sup(x)<\alpha$, and an increasing enumeration of $x$
has length less than $\alpha$.  By $<\alpha$-distributivity, this
enumeration is a ground-model sequence.  Thus $x\in V$.
Given any condition containing $v$, apply the fixed-seed end
extension in \cref{lem:targeted-insertion} with
$\theta>\sup(x)$.  The construction uses one chain per old top
node, so it respects the Kurepa bound.  The resulting condition has a top node
$b_q(v)\in\ell_q(\delta)$ with $x\subseteq\delta<\alpha$.
Since $\delta\in\PP_\alpha\alpha$, the set of conditions
witnessing $B^{\mathrm K}_\alpha(v)\restriction\delta\in
L^{\mathrm K}_\alpha$ above $x$ is dense below every condition
containing $v$.  This proves the claimed $\alpha$-cofinality.

Apply \cref{lem:small-cover} with $\rho=\alpha$ to the $\alpha$
pairwise distinct designated branches.  Any family generating
$L^{\mathrm K}_\alpha$ has size at least $\alpha$.  Conversely,
for each nonempty value of the generic ladder-coordinate function,
its top node has a designated branch extending every member of
that value.  Since the function has at most $\alpha$ coordinates,
$\alpha$ branches suffice.  This gives the asserted equality.
\end{proof}

Under the Kurepa bound, every condition level has cardinality less than
$\alpha$, and both the trace result and the exact covering number remain valid.
The companion forcing imposes only this ordinal-level bound and leaves the
low-cofinality ladder-coordinate convention unchanged.  The term
``Kurepa-bounded'' refers to the level-size inequality and does not assert
that the generic tree has $\alpha^+$ cofinal branches.  For $\alpha>\omega_2$,
this bound also does not by itself ensure that restriction levels have cardinality
less than $\omega_2$, as required for a $\PP_{\omega_2}\alpha$-tree.

\Cref{prop:uniform-diagonal} supplies the lower bound in
\cref{thm:exact-cover} independently of the stationary-trace argument.
By the maximum principle, a generating family of cardinality less than
$\alpha$ can be named by a sequence of that length.  The proposition
gives a member of $L_\alpha$ missed by the entire sequence.  The empty
family is ruled out by \cref{lem:targeted-insertion}.

\begin{proposition}\label{prop:uniform-diagonal}
Let $\alpha\geq\omega_2$ be regular, and let $\lambda<\alpha$ be a cardinal.
Suppose that
$p\in\Sstar(\alpha)$ and
$\langle\dot c_\xi:\xi<\lambda\rangle$ is a sequence of
$\Sstar(\alpha)$-names such that
\[
   p\forces
   ``\dot c_\xi\text{ is a cofinal branch through }T_\alpha''
   \qquad(\xi<\lambda).
\]
There exist $q\leq p$, $\delta\in\domn(\ell_q)$, and
$y\in\ell_q(\delta)$ such that
\[
   q\forces \check y\not\subseteq\dot c_\xi
   \qquad(\xi<\lambda).
\]
\end{proposition}

\begin{proof}
Choose an infinite cardinal $\mu<\alpha$ with $2^\mu>\lambda$ in
the ground model.  For infinite $\lambda$, one may take $\mu=\lambda$.
Apply \cref{lem:wide-level} below $p$ with this $\mu$ to obtain
$p_0\leq p$ and a level $\gamma$ containing a ground-model family
$U\subseteq(t_{p_0})_\gamma$ of more than $\lambda$ distinct nodes.
The cardinal preservation in \cref{thm:closure} ensures that this
inequality remains true in the extension.

Since the $\gamma$-th level is fixed by $p_0$, the condition
$p_0$ forces each $\dot c_\xi\restriction\gamma$ to belong to the
ground-model set
$(t_{p_0})_\gamma$.  The $<\alpha$-distributivity of
\cref{thm:closure} then supplies a single $r\leq p_0$ and nodes
$v_\xi\in(t_{p_0})_\gamma$ such that
\[
   r\forces\dot c_\xi\restriction\gamma=\check v_\xi
   \qquad(\xi<\lambda).
\]
Choose $u_*\in U\setminus\{v_\xi:\xi<\lambda\}$.
Apply \cref{lem:targeted-insertion} below $r$ with seed $u_*$ and
$\theta=\gamma$.  Obtain $q\leq r$, a top height $\delta$, and a
node $x=b_q(u_*)\in\ell_q(\delta)$.  The node $x$ extends $u_*$,
while every $\dot c_\xi$ is forced to have the different prefix
$v_\xi$ on level $\gamma$.  Hence
\[
   q\forces\check x\not\subseteq\dot c_\xi
   \qquad(\xi<\lambda).
\]
Taking $y=x$ proves the proposition.
\end{proof}

The scaled two-cardinal analysis rests on a comparison between cofinal branches of the ordinal
tree and branches of its restriction system.  More generally, suppose that
$\alpha$ is a regular cardinal in the ambient universe and that
$T\subseteq{}^{<\alpha}2$ is a normal binary tree of height $\alpha$.
For $x\in\PP_\alpha\alpha$, put
\[
   \beta_x=\sup(x)+1,
   \qquad
   \mathcal T_x=\{r\restriction x:r\in T_{\beta_x}\}.
\]
Regularity of $\alpha$ ensures $\beta_x<\alpha$.  The successor is included
so that $x\subseteq\beta_x$ also when $x$ has a maximum.  Normality of $T$
gives nonempty levels, and downward closure gives coherent restriction maps.
Thus the displayed system is a $\PP_\alpha\alpha$-tree exactly when each
level has size less than $\alpha$.

For the generic tree, \cref{thm:closure} ensures that $\alpha$ is a regular
cardinal in $V[G]$.  Every domain occurring in $L_\alpha$ therefore belongs to
$\PP_\alpha\alpha$.  If $y\in L_\alpha$ has ordinal domain $\xi$, then
$y\in(T_\alpha)_\xi$.  The node $y$ itself witnesses
$y\in\mathcal T_\xi$ when $\beta_\xi=\xi$, and normality supplies an extension
on level $\beta_\xi$ when $\beta_\xi>\xi$.  If
$y=x\restriction z$ arises from a ladder coordinate $\xi$ of uncountable
cofinality, then $z\in\PP_{\omega_1}\xi$ and hence $\beta_z<\xi$.
Downward closure gives
$x\restriction\beta_z\in(T_\alpha)_{\beta_z}$, which witnesses
$y\in\mathcal T_z$.  Thus $L_\alpha$ is a set of nodes of the restriction system.  By
\cref{prop:level-domain-failure}, it is not a ladder system on that restriction
system in the full sense of
\cite[Definition~4.2, p.~1119]{HayutMagidor2022}, even when the level-size
condition holds.

\begin{lemma}\label{lem:branch-comparison}
Let $\mathcal T$ be the restriction system just defined from $T$.  A
function $c\in{}^\alpha2$ is a cofinal branch through
the ordinal tree $T$ if and only if
\[
   c\restriction x\in\mathcal T_x
   \qquad\text{for every }x\in\PP_\alpha\alpha.
\]
Consequently, whenever the level-size condition holds, so that $\mathcal T$
is a $\PP_\alpha\alpha$-tree, its branches are exactly the cofinal branches
through $T$.
\end{lemma}

\begin{proof}
If every ordinal initial segment of $c$ belongs to $T$, then for
$x\in\PP_\alpha\alpha$ the node $c\restriction\beta_x$ belongs to
$T_{\beta_x}$ and witnesses $c\restriction x\in\mathcal T_x$.
Conversely, suppose that $c$ is a branch through $\mathcal T$.  For every
ordinal $\eta<\alpha$, $\eta\in\PP_\alpha\alpha$.  The branch
condition at the level $x=\eta$ gives $r\in T_{\beta_\eta}$ with
$c\restriction\eta=r\restriction\eta$.  Since
$\eta\leq\beta_\eta$ and $T$ is downward closed, it follows that
$c\restriction\eta\in T_\eta$.  Hence $c$ is a cofinal branch through
$T$.
\end{proof}

\begin{proposition}\label{prop:two-cardinal-boundary}
Let $\alpha\geq\omega_2$ be regular and let
$G\subseteq\Sstar(\alpha)$ be generic.  For the scaled restriction system
$\mathcal T_\alpha$ indexed by $\PP_\alpha\alpha$ and induced by $T_\alpha$, the following conditions are equivalent.
\begin{enumerate}[label=\textup{(\alph*)},leftmargin=2.2em]
  \item every level of $\mathcal T_\alpha$ has size less than $\alpha$.
  \item $\alpha$ is a strong limit cardinal in $V$.
  \item $\alpha$ is strongly inaccessible in $V$.
\end{enumerate}
When these conditions hold, $\mathcal T_\alpha$ is a
$\PP_\alpha\alpha$-tree and
\[
    \bcov(\mathcal T_\alpha,L_\alpha)=\alpha.
\]
In particular, if $\alpha$ is a successor cardinal in $V$, then
$\mathcal T_\alpha$ is not a $\PP_\alpha\alpha$-tree.
\end{proposition}

\begin{proof}
Assume first that every level of $\mathcal T_\alpha$ has size less than
$\alpha$.  Let $\mu<\alpha$ be an infinite cardinal in $V$.  By
\cref{lem:wide-level} and genericity, there are a limit ordinal
$\delta<\alpha$ and a ground-model injection
\[
   ({}^\mu2)^V\longrightarrow(T_\alpha)_\delta.
\]
Since $\delta$ is a limit ordinal, $\beta_\delta=\delta+1$.  Normality and
downward closure of $T_\alpha$ give
\[
   \mathcal T_\delta
      =\{r\restriction\delta:r\in(T_\alpha)_{\delta+1}\}
      =(T_\alpha)_\delta.
\]
The ground-model injection remains an injection in $V[G]$.  Because $\alpha$
remains a cardinal there, the inequality
$(2^\mu)^V\geq\alpha$ would contradict
$|\mathcal T_\delta|<\alpha$.  Hence
\[
   (2^\mu)^V<\alpha
   \qquad\text{for every infinite cardinal }\mu<\alpha\text{ in }V,
\]
so $\alpha$ is a strong limit cardinal in $V$.

Conversely, suppose that $\alpha$ is a strong limit cardinal in $V$.  Let
$x\in\PP_\alpha\alpha$ in $V[G]$.  Regularity of $\alpha$ in the extension
gives $\beta_x<\alpha$.  Choose $p\in G$ whose tree height is above
$\beta_x$.  End extension fixes that level, so
\[
   (T_\alpha)_{\beta_x}=(t_p)_{\beta_x}.
\]
The ordinal $\beta_x$ belongs to $V$, and $|\beta_x|^V<\alpha$.  In the
ground model put
\[
   \kappa=\bigl|{}^{\beta_x}2\bigr|^V
      =\bigl(2^{|\beta_x|^V}\bigr)^V<\alpha
\]
and fix an injection
\[
   (t_p)_{\beta_x}\longrightarrow\kappa.
\]
The same injection exists in $V[G]$.  Since $\kappa<\alpha$ and $\alpha$
is a cardinal there, $|\kappa|^{V[G]}<\alpha$.  The level $\mathcal T_x$ is
the image of $(t_p)_{\beta_x}$ under restriction to $x$, and hence
$|\mathcal T_x|<\alpha$.  This proves
the equivalence of (a) and (b).  Since $\alpha$ is regular and uncountable in
$V$, conditions (b) and (c) are
equivalent.

Under these conditions, the restriction system is a
$\PP_\alpha\alpha$-tree.  By \cref{lem:branch-comparison}, its branches are
exactly the cofinal branches through $T_\alpha$, so the exact
covering number follows from \cref{thm:exact-cover}.  A successor cardinal in $V$ is not a
strong limit cardinal in $V$, which gives the final assertion.
\end{proof}

\begin{proof}[Proof of \cref{thm:main}]
The preservation assertion is \cref{thm:closure}.  Part~\textup{(i)}
is \cref{cor:bistationary}, part~\textup{(ii)} is
\cref{thm:exact-cover}, part~\textup{(iii)} follows from
\cref{lem:wide-level} and genericity, and part~\textup{(iv)} is
\cref{cor:hm-restriction-wide}.  \Cref{prop:two-cardinal-boundary} gives the exact scaled restriction-system boundary
and the accompanying exact branch-covering number.
\end{proof}

For every regular $\alpha\geq\omega_2$, the traces with respect to
$L_\alpha$ of any nonempty family of fewer than $\alpha$ designated generic
branches have a bistationary intersection in $\PP_{\omega_1}\alpha$.
The same trace conclusion holds for the Kurepa-style companion
forcing.  In the respective generic extensions, each version has exact
branch-covering number $\alpha$.
In the unrestricted forcing, for every infinite ground-model cardinal
$\mu<\alpha$, some level of the generic tree contains a copy of $({}^\mu2)^V$.
The endpoint-corrected restriction family indexed by $\PP_{\omega_2}\alpha$
has a level of size at least $\omega_2$ and hence is not a
$\PP_{\omega_2}\alpha$-tree.  For the scaled $\PP_\alpha\alpha$ system, every level has size less than
$\alpha$ exactly when $\alpha$ is strongly inaccessible in $V$, and in that
case the branch-covering number of $L_\alpha$ relative to that system equals
$\alpha$.  A separate failure comes from the low-cofinality empty-value
convention.  It prevents the set of domains of $L_\alpha$ from containing a
club, so $L_\alpha$ is not a ladder system in the full two-cardinal sense.  The
forcing $\Sstar(\alpha)$ is $\sigma$-closed and, after adjoining a formal
maximum, $\alpha$-strategically closed.  It is therefore
$<\alpha$-distributive and preserves $\alpha$ as a regular cardinal.  The
conclusions stated here for $\Sstar(\alpha)$ concern its direct generic
extension.  No forcing equivalence
with the printed Hayut--Magidor presentation is claimed, and the results do
not settle Hayut and Magidor's question of whether cofinal catching can be
separated from stationary catching.

At $\alpha=\omega_2$, the Kurepa level bound makes the
endpoint-corrected restriction system induced by
$T^{\mathrm K}_{\omega_2}$ a $\PP_{\omega_2}\omega_2$-tree,
since each restriction level has size at most $\omega_1$.
The corresponding ladder-coordinate set $L^{\mathrm K}_{\omega_2}$
nevertheless fails the club-many-level requirement for a ladder system
on that restriction tree, by the same argument as in
\cref{prop:level-domain-failure}.
It remains to determine whether a typed modification can overcome
this obstruction while preserving the common-trace and exact-cover
conclusions.

A further question is whether the master-condition argument in the proof of
Lemma~5.11 \cite[p.~1126]{HayutMagidor2022} admits a typed counterpart.  Relevant forcing background includes Mitchell's original tree-property forcing
\cite{Mitchell1972}, the iterated tree-property constructions of Cummings
and Foreman \cite{CummingsForeman1998}, and Krueger's general treatment of mixed-support
iterations \cite{Krueger2008}.  For lifting and strong master conditions, see
\cite[Proposition~9.1, p.~805, and Definition~12.2, p.~814]{Cummings2010}.  A separate problem is to determine which
subsequent forcing notions preserve the trace conclusions.  For related work
on the preservation and destruction of the tree property under further
forcing, see \cite{Unger2012,Unger2015}.

\medskip
\noindent\textbf{Acknowledgments.}
The author is grateful to Yair Hayut and Menachem Magidor, whose work on ladder systems motivated the questions studied here.

\end{document}